\documentclass[final,1p,times]{elsarticle}

\usepackage[colorlinks=true]{hyperref}
\usepackage{mathrsfs}
\usepackage{amsthm,amsmath,amssymb,mathrsfs,amsfonts,graphicx,graphics,latexsym,exscale,cmmib57,dsfont,amscd,ulem}
\usepackage{color,soul}
\newtheorem*{BH}{Bergelson-Hindman Theorem}
\newtheorem*{Gr}{Gr\"{u}nwald}
\newtheorem*{FW}{Multiple Birkhoff Recurrence Theorem}
\newtheorem*{ope-que}{Question}
\newtheorem{thm}{Theorem}[section]
\newtheorem{con}[thm]{Conjecture}
\newtheorem{lem}[thm]{Lemma}
\newtheorem{cor}[thm]{Corollary}
\newtheorem{cor-1.1}{Corollary~1.1.\!}
\newtheorem{cor-1.2}{Corollary~1.2.\!}
\newtheorem{cor-4.1}{Corollary}
\newtheorem{prop}[thm]{Proposition}
\newtheorem{que}[thm]{Question}

\theoremstyle{definition}
\newtheorem{defn}[thm]{Definition}
\newtheorem{deff}{Basic Notion}
\newtheorem{note1.2}{Note~\ref{thm1.2}.\!}
\newtheorem*{note}{Note}
\newtheorem*{sn}{Standing notation}

\numberwithin{equation}{section}

\newcommand{\e}{\boldsymbol{o}}
\journal{IJM}

\begin{document}

\begin{frontmatter}
\title{Gr\"{u}nwald version of van der Waerden's theorem for semi-modules}
\author{Xiongping Dai}
\ead{xpdai@nju.edu.cn}
\address{Department of Mathematics, Nanjing University, Nanjing 210093, People's Republic of China}

\begin{abstract}
Let $(M,\pmb{+})$ be any semi-module over a semi-ring $(R,+,\cdot)$ with a finite coloring $M=B_1\cup\dotsm\cup B_q$.
By establishing a Regional Multiple Recurrence Theorem for semi-modules, we prove that one of the colors $j$ has the property that if $F\subseteq M$ is any finite set, then one can find some ``syndetic'' subset $D_F$ of $(R,+)$ such that for each $d\in D_F$ there is some $a\in B_j$ with $a\pmb{+}dF\subseteq B_j$. This in turn implies that each uniformly almost periodic point is multiply almost periodic.
\end{abstract}

\begin{keyword}
Van der Waerden theorem $\cdot$ Gr\"{u}nwald theorem $\cdot$ Multiple recurrence

\medskip
\MSC[2010] 11B25 $\cdot$ 37B20 $\cdot$ 05E15
\end{keyword}
\end{frontmatter}

\section{Introduction}\label{sec1}
\subsection{Van der Waerden theorems}\label{sec1.1}
B.\;L.\,Van der Waerden's Theorem, conjectured by Baudet and proved in 1927, states (in one of several equivalent formulations) that if $\mathbb{N}=\{1,2,3,\dotsc\}$ is partitioned into finitely many sets, say $\mathbb{N}=B_1\cup\dotsm\cup B_q$, then one of these sets $B_j$ contains arithmetic progressions of arbitrary finite length (cf.~\cite{Wae,Fur}).

Since for any finite set $F\subset\mathbb{N}$ there is some $l\ge1$ such that $F\subseteq\{1,2,\dotsc,l\}$, hence van der Waerden's theorem is equivalent to the $1$-dimensional case of Gr\"{u}nwald's Theorem:

\begin{Gr}[\cite{Rad, Fur}]
Let $\mathbb{N}^m=B_1\cup\dotsm\cup B_q$ be an arbitrary finite partition of the $m$-dimensional positive lattice $\mathbb{N}^m$, where $1\le m<\infty$. Then one of the sets $B_j$ has the property that if $F\subset\mathbb{N}^m$ is any finite set, then $B_j$ contains a translate of a dilation of $F$: $c+bF\subset B_j$ where $c\in\mathbb{N}^m,b\in\mathbb{N}$.
\end{Gr}

Many extensions of Gr\"{u}nwald's theorem have been made since Furstenberg 1981; see, e.g., \cite{BL, BFM} for polynomial extensions of $b\in\mathbb{N}$. Another direction is for extensions of $\mathbb{N}^m$ to semigroups setup.

Let $(G,\pmb{+})$ be any nontrivial additive semigroup; then an analogue of van der Waerden's Theorem holds trivially by van der Waerden's Theorem itself. Indeed, let $g\in G$ be an arbitrary nonzero element, we set $\widehat{\mathbb{N}}=\{ng\,|\, n\in\mathbb{N}\}$ where $ng=g\pmb{+}\dotsm\pmb{+}g$ ($n$ times), and define a homomorphism $\varphi\colon n\mapsto ng$ from $\mathbb{N}$ onto $\widehat{\mathbb{N}}$. If $G=B_1\cup\dotsm\cup B_q$ is any finite partition of $G$ and let $\widehat{B}_j=\{n\,|\,ng\in B_j\}$, then
$\mathbb{N}=\widehat{B}_1\cup\dotsm\cup \widehat{B}_q$
is a finite partition of $\mathbb{N}$ so that some $\widehat{B}_j $ contains $(l+1)$-length arithmetic progressions $\{a,a+d,\dotsc,a+ld\}$ for all $l\ge1$ by van der Waerden's Theorem. Via $\varphi$ this implies that
$B_j$ contains arithmetic progressions $\{x,x\pmb{+}y,\dotsc,x\pmb{+}ly\}$ where $x=ag$ and $y=dg$ of every finite length $l+1$ as well.

However to such a version of van der Waerden's theorem there is no guarantee in general that the ``common difference'' $y$ of $\{x,x\pmb{+}y,\dotsc,x\pmb{+}ly\}$ will not be the zero element of the additive semigroup $(G,\pmb{+})$. To avoid such triviality, using the Stone-\v{C}ech compactification of discrete semigroup and ultrafilter methods, as a result of their Central Sets Theorem, Bergelson and Hindman in 1992 proved a strengthened version of van der Waerden's theorem as follows:

\begin{BH}[{\cite[Corollary~3.2]{BH}}]
Let $(G,\pmb{+})$ be an abelian cancelable semigroup, let $\langle d_m\rangle_1^\infty$ be a sequence in $G$ with $d_m\not=d_n$ for $m\not=n$, and let $G=B_1\cup\dotsm\cup B_q$. Then there exists some $B_j$ such that to any $l\in\mathbb{N}$, one can find $a\in G$ and $d\in \textrm{FS}(\langle d_m\rangle_1^\infty)$ with $d\not=0$ such that $a, a\pmb{+}d, \dotsc,a\pmb{+}ld\in B_j$. Here $\textrm{FS}(\langle d_m\rangle_1^\infty)=\{d_{n_1}\pmb{+}\dotsm\pmb{+}d_{n_k}\,|\,1\le n_1<\dotsm<n_k<\infty, k\ge1\}$ is the IP-set generated by $\langle d_m\rangle_1^\infty$.
\end{BH}

Since an abelian semigroup $(G,\pmb{+})$ is just a semi-module over the semi-ring $(\mathbb{N},+,\cdot)$ or $(\mathbb{Z}_+,+,\cdot)$, we will further generalize Gr\"{u}nwald's version from $\mathbb{Z}$ or $\mathbb{N}$ to any semi-modules in this paper.

\medskip
Recall that by a \textit{semi-ring} $(R,+,\cdot)$, it means a nonempty set $R$, together with two laws of composition called \textit{addition} $+$ and \textit{multiplication} $\cdot$ respectively, satisfying the following axioms:
\begin{itemize}
\item $(R,+)$ is an \textit{abelian} semigroup with zero element $0$;\footnote{If $(R,+,\cdot)$ itself does not have the zero element then we will adjoint $0$ to it by letting $0+t=t+0=t$ and $0\cdot t=t\cdot 0=0$ for all $t\in R$, and further consider $(R\cup\{0\},+,\cdot)$ instead of $(R,+,\cdot)$.}
\item $(R,\cdot)$ is a semigroup (not necessarily commutative) with \textit{unit} element $1$, which is \textit{associative}:
$(x\cdot y)\cdot z=x\cdot(y\cdot z)\; \forall x,y,z\in R$ and which is such that $0\cdot x=0\ \forall x\in R$;

\item $(R,+,\cdot)$ is \textit{distributive}:
\begin{gather*} (x+y)\cdot z=x\cdot z+y\cdot z\quad\textrm{and}\quad z\cdot(x+y)=z\cdot x+z\cdot y\end{gather*}
for all $x,y,z\in R$.
\end{itemize}

\begin{deff}\label{deff1}
A subset $S$ of a semi-ring $(R,+,\cdot)$ is called ``syndetic'' or ``relatively dense'' in $(R,+)$ if one can find a finite set $K\subseteq R$ such that
\begin{gather*}
(K+t)\cap S\not=\emptyset\quad \forall t\in R.
\end{gather*}
See, e.g., \cite{Fur, CD}. It is different with requiring $R=K+S$ by \cite[Definition~2.02]{GH}.
\end{deff}

Since here $R$ is provided with the discrete topology, so the notion of syndetic is the strongest one.

\begin{deff}
Let $(R,+,\cdot)$ be a semi-ring. As usual a \textit{semi-module over $R$} or an \textit{$R$-semimodule} $(M,\pmb{+})$ is an abelian semigroup with zero element $\e$, usually written additively, together with a scalar multiplication $(t,g)\mapsto tg$ of $R$ on $M$ such that
\begin{gather*}
(r+t)g=rg\pmb{+}tg\quad \textrm{and}\quad r(g\pmb{+}h)=rg\pmb{+}rh\quad \forall r,t\in R, g,h\in M
\end{gather*}
and
\begin{gather*}
1g=g\quad\textrm{and}\quad 0g=\e\quad\forall g\in M.
\end{gather*}
See \cite{Lan}.
In a similar way, one can define a \textit{right} $R$-semimodule via $(g,t)\mapsto gt$ from $M\times R$ to $M$. We shall deal only with left $R$-semimodules, unless specified otherwise.
\end{deff}

Clearly, $(\mathbb{R}^m,\pmb{+})$ is an $(\mathbb{R},+,\cdot)$-module, $(\mathbb{Z}_+^m,\pmb{+})$ is a $(\mathbb{Z}_+,+,\cdot)$-semimodule,
$(\mathbb{Z}_p^m,\pmb{+})$ is a module over the $p$-adic integer ring $(\mathbb{Z}_p,+,\cdot)$, and $(\mathbb{Q}_p^m,\pmb{+})$ is a module over the $p$-adic number field $(\mathbb{Q}_p,+,\cdot)$; cf., e.g.,~\cite{Lan}.

In the present paper we will mainly prove the following more general generalization of Gr\"{u}nwald's Theorem.

\begin{thm}\label{thm1.1}
Let $M=B_1\cup\dotsm\cup B_q$ be an arbitrary finite partition of a semi-module $(M,\pmb{+})$ over a semi-ring $(R,+,\cdot)$. Then one of the sets $B_j$ has the property that if $F\subseteq M$ is any finite set, then one can find a syndetic subset $D_F$ of $(R,+)$ such that for each $d\in D_F$ there is an $a\in B_j$ with $a\pmb{+}dF\subseteq B_j$.
\end{thm}

\begin{note}
If $M$ is a right $R$-semimodule and we require $a\pmb{+}Fd\subseteq B_j$ instead of $a\pmb{+}dF\subseteq B_j$, then the statement holds as well.
In addition, Theorem~\ref{thm1.1} has a ``finitary formulation''; see Theorem~\ref{thm3.9} in $\S\ref{sec3.2}$ below.
\end{note}

This theorem is not subsumed by the above theorem of Bergelson and Hindman, because here every $R$-semimodule $(M,\pmb{+})$ does not need to have a sequence $\langle d_m\rangle$ as in Bergelson and Hindman's statement satisfying that for any finite subset $F$ of $M$ and some ``syndetic'' $d\in R$,
\begin{gather*}
a\pmb{+}dF\subseteq\{a,a\pmb{+}d^\prime,\dotsc,a\pmb{+}ld^\prime\}
\end{gather*}
for some $a\in B_j, d^\prime\in\textsl{FS}(\langle d_m\rangle)$ and $l\ge1$.

A direct consequence of Theorem~\ref{thm1.1} is the following theorem of van der Waerden type for some canonical modules:

\begin{cor-1.1}
Let $M=\mathbb{R}^m$ (resp.~$\mathbb{Q}^m,\mathbb{Z}_p^m,\mathbb{Q}_p^m$) and $B_1\cup\dotsm\cup B_q$ be any finite partition of $M$. Then one of the sets $B_j$ has the property that if $F$ is a finite subset of $M$, then there are two elements $a\in B_j$ and $d\in R=\mathbb{R}$ (resp.~$\mathbb{Q},\mathbb{Z}_p,\mathbb{Q}_p$) with $d\not=0$ such that $a\pmb{+}dF\subset B_j$.
\end{cor-1.1}
\noindent
Here $\mathbb{R}$ is the real field, $\mathbb{Q}$ the rational field, $\mathbb{Z}_p$ the $p$-adic integer ring and $\mathbb{Q}_p$ the $p$-adic number field. More generally, from Theorem~\ref{thm1.1} follows

\begin{cor-1.1}
Let $G=B_1\cup\dotsm\cup B_q$ be any finite partition of an abelian semigroup $(G,\pmb{+})$. Then one of the sets $B_j$ has the property that if $F$ is a finite subset of $G$ one can find a syndetic subset $D_F$ of $\mathbb{Z}_+$ such that for each $d\in D_F$ there is an $a\in B_j$ with $a\pmb{+}dF\subseteq B_j$.
\end{cor-1.1}

Following the framework of \cite{BFK}, to prove Theorem~\ref{thm1.1} we will need to prove a regional multiple recurrence theorem (cf.~Theorem~\ref{thm1.2} below). Theorem~\ref{thm1.1} and the later Theorem~\ref{thm1.2} are in fact equivalent to each other by the following

\begin{cor-1.1}\label{cor1.1-3}
Let $(M,\pmb{+})$ be a semi-module over a semi-ring $(R,+,\cdot)$, and let $\varphi\colon M\times X\rightarrow X$ be a discrete semiflow on an arbitrary set $X$; that is to say,
\begin{gather*}
\varphi(\e,x)=x,\ \varphi(f\pmb{+}g,x)=\varphi(f,\varphi(g,x))\quad \forall f,g\in M\textrm{ and }x\in X.
\end{gather*}
Then for any finite partition of $X=X_1\cup\dotsm\cup X_q$, there exists a cell $X_\alpha$ such that for all elements $T_1,\dotsc,T_l\in M$ one can find a syndetic set $D$ of $(R,+)$ with
\begin{gather*}
X_\alpha\cap\varphi^{-dT_1}[X_\alpha]\cap\dotsm\cap\varphi^{-dT_l}[X_\alpha]\not=\emptyset\quad \forall d\in D.
\end{gather*}
Here $\varphi^{-g}[A]=\{x\in X\,|\,\varphi(g,x)\in A\}$ for all $g\in M$ and $A\subseteq X$.
\end{cor-1.1}

\begin{proof}
We pick any point $x\in X$ and let $B_j=\{g\in M\,|\,\varphi(g,x)\in X_j\}$ for all $j=1,2,\dotsc,q$. Then $M=B_1\cup\dotsm\cup B_q$ is a finite partition of $M$, and by Theorem~\ref{thm1.1} it follows that some cell $B_\alpha$ satisfies that to any finite set $F=\{\e,T_1,\dotsc,T_l\}\subset M$, there is a syndetic set $D=D_F$ of $(R,+)$ so that for all $d\in D$, one can find some $a\in M$ with $a\pmb{+}dF\subseteq B_\alpha$.
This implies that
\begin{gather*}
\varphi(a,x)\in X_\alpha\cap\varphi^{-dT_1}[X_\alpha]\cap\dotsm\cap\varphi^{-dT_l}[X_\alpha]
\end{gather*}
as desired. The proof is completed.
\end{proof}

The opposite implication ``Theorem~\ref{thm1.2} $\Rightarrow$ Theorem~\ref{thm1.1}'' will be obtained by using Furstenberg's correspondence principle (cf. Proof of Theorem~\ref{thm1.1} in $\S\ref{sec3.2}$).

\subsection{Topological multiple recurrence theorems}\label{sec1.2}
Let $X$ be any topological space and $(M,\pmb{+})$ a semi-module over a semi-ring $(R,+,\cdot)$; whenever the action mapping $\varphi\colon M\times X\rightarrow X$ of $M$ from the left on $X$
is such that:
\begin{itemize}
\item $\varphi(\e,x)=x\ \forall x\in X$, i.e., $\varphi(\e,\centerdot)=i_X$;
\item $\varphi(g,\centerdot)\colon X\rightarrow X$, for all $g\in M$, is continuous; and

\item $\varphi(g\pmb{+}h,x)=\varphi(g,\varphi(h,x))$, i.e. $\varphi^{g\pmb{+}h}=\varphi^g\circ\varphi^h$ if we agree to write $\varphi^g(x)$ in place of $\varphi(g,x)$, for all $g,h\in M$;\footnote{Since we have required $(M,\pmb{+})$ abelian, hence now $\varphi(g\pmb{+}h,x)=\varphi(h\pmb{+}g,x)$, i.e., $\varphi^h\varphi^g=\varphi^g\varphi^h$. However, when $(M,\pmb{+})$ is a non-abelian discrete semigroup, then $\varphi^h\varphi^g\not=\varphi^g\varphi^h$ in general.}
\end{itemize}
then we shall call $\varphi\colon M\times X\rightarrow X$ a \textit{semiflow} on $X$, denoted by $(\varphi,M,X)$. Clearly, there is no use of the commutativity of $M$ for this notation itself.

\begin{sn}
Let $(\varphi,M,X)$ be a semiflow. Given any point $x\in X$, $M_\varphi[x]=\varphi(M,x)$ is called the \textit{orbit} of the motion $\varphi(\centerdot,x)$; write $M[x]$ if no confusion. For $g\in M$, $\varphi^{-g}\colon X\rightarrow X$ will be defined as the inverse map $\left(\varphi^g\right)^{-1}$, possibly multivalent, of the $g$-sample map $\varphi(g,\centerdot)\colon X\rightarrow X$. Of course, if $(M,\pmb{+})$ is a group, then the $-g$ is just the inverse element of $g$ in $\varphi^{-g}$ here.

$(\varphi,M,X)$ is \textit{minimal} if and only if there does not exist a nonempty proper closed subset $Y$ of $X$ such that $M_\varphi[Y]\subseteq Y$. Similarly one can define \textit{minimal subset} of $(\varphi,M,X)$.

An $x\in X$ is an \textit{almost periodic} (a.p.) \textit{point} of $(\varphi,M,X)$ if and only if $\mathrm{cls}_XM_\varphi[x]$ is a minimal subset of $X$. This can be extended to any semigroup $M$ acting on a compact Hausdorff space.
\end{sn}

We can then obtain the following ``Regional Multiple Recurrence Theorem'' which generalizes and strengthens the classical theorem of Furstenberg and Weiss 1978~\cite[Theorem~1.5]{FW} that is for the invertible case of $M=\mathbb{Z}^m$ over the canonical ring $(\mathbb{Z},+,\cdot)$ where $M$ acts on a compact metric space. However, our phase space $X$ is not necessarily metrizable here. This point is important for our subsequent applications.

\begin{thm}\label{thm1.2}
Given any semi-module $(M,\pmb{+})$ over a semi-ring $(R,+,\cdot)$, let $(\varphi,M,X)$ be a minimal semiflow on a compact Hausdorff space $X$. Then for all $T_1,\dotsc,T_l\in M$ and any open subset $U$ of $X$, $U\not=\emptyset$, the multiple hitting-time set of $U$ with itself,
\begin{gather*}
N_{T_1,\dotsc,T_l}(U):=\left\{t\in R\,|\,U\cap\varphi^{-tT_1}[U]\cap\dotsm\cap\varphi^{-tT_l}[U]\not=\emptyset\right\},
\end{gather*}
is syndetic in $(R,+)$.
\end{thm}

\begin{note1.2}
The statement still holds if instead $M$ is a right semi-module by defining the hitting-time set as follows:
\begin{gather*}
N_{T_1,\dotsc,T_l}(U):=\left\{t\in R\,|\,U\cap\varphi^{-T_1t}[U]\cap\dotsm\cap\varphi^{-T_lt}[U]\not=\emptyset\right\},
\end{gather*}
for all $T_1,\dotsc,T_l\in M$ and any nonempty open subset $U$ of $X$.
\end{note1.2}

\begin{note1.2}
In 1989~\cite{BPT} Blaszczyk et al. proved that if $M$ is an abelian semigroup (i.e. $M$ is a semi-module over $(\mathbb{Z}_+,+,\cdot)$) and it acts minimally on $X$, then for any non-empty open set $U$ of $X$ and $T_1,\dotsc,T_l\in M$ there exists an integer $n\ge1$ with $U\cap T_1^{-n}[U]\cap\dotsm\cap T_l^{-n}[U]\not=\emptyset$. Also see \cite[Theorem~6.14]{EEN} for actions of infinite abelian group by entirely different approaches.
\end{note1.2}

\begin{note1.2}
Moreover, the ``syndeticity'' of $N_{T_1,\dotsc,T_l}(U)$ did not appear in the original work of Furstenberg and Weiss in 1978 nor in the work of Blaszczyk et al. in 1989 for $\mathbb{Z}^m$ over $(\mathbb{Z},+,\cdot)$ (cf.~\cite[Theorem~1.5]{FW}, \cite[Theorem~1.56]{Gla}, and \cite[Theorem~1 and Corollary~2]{BPT}) nor in the work of Ellis et al. \cite{EEN}.
\end{note1.2}

\begin{note1.2}
Theorem~\ref{thm1.2} is never a consequence of the Furstenberg-Katznelson Multiple Poincar\'{e} Recurrence Theorem (\cite{FK, Fur}); this is because if $(R,+,\cdot)$ is uncountable, then we cannot reduce a Baire probability space $(X,\mathcal{B}a(X),\mu)$ to a standard Borel probability space by a factor map (to employ the disintegration of measures).
\end{note1.2}

\begin{note1.2}
If $M$ is not abelian, then the Regional Multiple Recurrence Theorem need not hold. For example, take $X=[0,1]$, $T_1x=x/2$, $T_2x=(1+x)/2$ (cf.~\cite[p.~40]{Fur}).
\end{note1.2}

Here for any $t\in R$ and any open set $U\subset X$, we have that
\begin{gather*}
U\cap\varphi^{-tT_1}[U]\cap\dotsm\cap\varphi^{-tT_l}[U]\not=\emptyset\Leftrightarrow \exists\, x_0\in U\textit{ s.t. } \varphi(tT_i,x_0)\in U\textrm{ for all }i=1,\dotsc,l.
\end{gather*}

In this paper Theorem~\ref{thm1.1} is a consequence of Theorem~\ref{thm1.2}. However, if we start with Theorem~\ref{thm1.1}, then Corollary~1.2.\ref{cor1.2.1} below follows at once from Corollary~1.1.\ref{cor1.1-3}, which in turn implies Theorem~\ref{thm1.2} by a standard homogeneity argument under the minimality hypothesis (see Note~\ref{thm1.2}.\ref{n1.2.6}).

\begin{cor-1.2}[Multiple Recurrence in Open Covers]\label{cor1.2.1}
Given any semi-module $(M,\pmb{+})$ over a semi-ring $(R,+,\cdot)$, let $(\varphi,M,X)$ be a semiflow on a compact Hausdorff space $X$, and let $\mathfrak{U}$ be an open cover of $X$. Then there exists some $U\in\mathfrak{U}$ such that $N_{T_1,\dotsc,T_l}(U)$ is syndetic in $(R,+)$ for all $T_1,\dotsc,T_l\in M$.
\end{cor-1.2}

\begin{proof}
Let $X_0$ be a minimal subset of $(\varphi,M,X)$. Since there is some $U\in\mathfrak{U}$ with $U\cap X_0\not=\emptyset$ and $N_{T_1,\dotsc,T_l}(U\cap X_0)\subseteq N_{T_1,\dotsc,T_l}(U)$, hence by Theorem~\ref{thm1.2} for the subsemiflow $(\varphi,M,X_0)$ it follows that $N_{T_1,\dotsc,T_l}(U)$ is syndetic in $(R,+)$. This proves the corollary.
\end{proof}

\begin{note1.2}\label{n1.2.6}
It turns out that this corollary implies Theorem~\ref{thm1.2} as follows. Let $U$ be an open nonempty subset of $X$. Since $(\varphi,M,X)$ is minimal and $X$ is compact, there are $f_1,\dotsc,f_n\in M$ such that $\varphi^{-f_1}[U]\cup\dotsm\cup\varphi^{-f_n}[U]=X$. Then by Corollary~\ref{thm1.2}.\ref{cor1.2.1}, $N_{T_1,\dotsc,T_l}(V)$ is syndetic in $(R,+)$ for $V=\varphi^{-f_j}[U]$ for some $j=1,\dotsc,n$. Now if $x=\varphi(f_j,y)$ with $y\in V$ and $t\in N_{T_1,\dotsc,T_l}(V)$such that $\varphi(tT_i,y)\in V$ for all $1\le i\le l$, then $\varphi(tT_i,x)=\varphi(tT_i,\varphi(f_j,y))=\varphi^{f_j}(\varphi(tT_i,y))\in U$. This thus shows that $N_{T_1,\dotsc,T_l}(U)\supseteq N_{T_1,\dotsc,T_l}(V)$ is syndetic in $(R,+)$.
\end{note1.2}

Next we will consider another application of Theorem~\ref{thm1.2} to pointwise multiple recurrence. For that we first need to recall and introduce some notions as follows:

\begin{defn}\label{def1.3}
Let $(\varphi,M,X)$ be a semiflow on a compact Hausdorff space $X$, where $(M,\pmb{+})$ be a semi-module over a semi-ring $(R,+,\cdot)$. By $\mathscr{U}_x$ it stands for the neighborhood system of $X$ at $x$, for all $x\in X$.
\begin{enumerate}
\item[(a)] Given $T\in M$, a point $p\in X$ is said to be \textit{almost periodic} (a.p.) for $(\varphi,M,X)$ relative to $T$ iff it an a.p. point of $\varphi_T\colon R\times X\rightarrow X$ by $(t,x)\mapsto tx:=\varphi(tT,x)$; i.e.,
\begin{gather*}
N_T(p,U)=\{t\in R\,|\,tp\in U\}\quad \forall U\in\mathscr{U}_p,
\end{gather*}
is syndetic in $(R,+)$ under the sense of Basic Notion~\ref{deff1}; see, e.g.,~\cite{GH, Fur, Aus, CD}.

\item[(b)] A point $p\in X$ is called \textit{multiply almost periodic} (m.a.p.) for $(\varphi,M,X)$ if it is a.p. uniformly for all $T_1,\dotsc,T_l\in M$ and $l\ge2$; i.e.,
    \begin{gather*}
    N_{T_1,\dotsc,T_l}(p,U)=\left\{t\in R\,|\,\varphi(tT_i,p)\in U,\ i=1,\dotsc, l\right\}\quad \forall U\in\mathscr{U}_p
    \end{gather*}
    is syndetic in $(R,+)$, for all $T_1,\dotsc,T_l\in M$ and $l\ge2$.\footnote{A cyclic system $(T,X)$ is called \textit{multi-minimal} in \cite{KO,CLL} if $(T\times T^2\times\dotsm\times T^n, X^n)$ is minimal for all $n\ge1$. In this case, every point $x\in X$ must be m.a.p. for $(\varphi,\mathbb{Z}_+,X)$ where $\varphi(g,x)=T^gx$. However, it does not need to imply the multi-minimality of $(T,X)$ that every point of $X$ is m.a.p. for $(\varphi,\mathbb{Z}_+,X)$.}

\item[(c)] A point $p$ is called \textit{uniformly almost periodic} (u.a.p.) for $(\varphi,M,X)$ if the orbit closure $\mathrm{cls}_X^{}M_\varphi[p]$ is minimal for $(\varphi,M,X)$ and $\{\varphi(T,\centerdot)\colon X\rightarrow X\}_{T\in M}$ is an equicontinuous family restricted to $\mathrm{cls}_X^{}M_\varphi[p]$ (cf.~\cite[Definition~V8.01, Theorem~V8.05]{NS} for $M=\mathbb{R}$).

\item[(d)] A point $p\in X$ is said to be \textit{multiply syndetically nonwandering} for $(\varphi,M,X)$ if for any $T_1,\dotsc,T_l\in M$ and any $U\in\mathscr{U}_p$, $N_{T_1,\dotsc,T_l}(U)$ is syndetic in $(R,+)$. This concept is stronger than the regionally recurrent point of $(\varphi,M,X)$ defined by Gottschalk and Hedlund~\cite[Remark~7.12]{GH}.
\end{enumerate}
\end{defn}

Although there does not need to exist m.a.p. points in general, yet we can easily obtain the following two statements from Theorem~\ref{thm1.2}.

\begin{cor-1.2}\label{cor-1.2-2}
If $p\in X$ is an u.a.p. point of $(\varphi,M,X)$, then it is m.a.p. for $(\varphi,M,X)$.
\end{cor-1.2}

\begin{proof}
Write $Y=\mathrm{cls}_X^{}M_\varphi[p]$ and then the subsemiflow $(\varphi,M,Y)$ is minimal and equicontinuous. Let $U\in\mathscr{U}_p$ be arbitrarily given. By Theorem~\ref{thm1.2}, for all $V\in\mathscr{U}_p$ and any $T_1,\dotsc,T_l\in M$, $N_{T_1,\dotsc,T_l}(V)$ is syndetic in $(R,+)$.

Since $\varphi\colon M\times Y\rightarrow Y$ is equicontinuous
and $Y$ is a compact Hausdorff space (so a uniform space), we can take some $V\subset U$ so ``small'' that
if $\varphi(T,y)\in V$ for some $y\in V$ and $T\in M$ then $\varphi(T,x)\in U\ \forall x\in V$.
Now for any $t\in N_{T_1,\dotsc,T_l}(V)$, there is some point $y_t\in V$ with $\varphi(tT_i,y_t)\in V$ simultaneously for $i=1,\dotsc,l$, and thus
$\varphi(tT_i,p)\in U$ simultaneously for $i=1,\dotsc,l$. This shows that $N_{T_1,\dotsc,T_l}(p,U)$ is syndetic in $(R,+)$ and thus $p$ is m.a.p. for $(\varphi,M,X)$.

The proof of Corollary~1.2.\ref{cor-1.2-2} is therefore completed.
\end{proof}

Since every semiflow $(\varphi,M,X)$ has a least invariant closed equivalence relation $S_e$ in $X$ such that the canonical factor $(\varphi,M,X/S_e)$ is an equicontinuous semiflow (cf.~\cite{EG, V}), hence the dynamics of m.a.p. is not rare.

\begin{cor-1.2}\label{cor-1.2-3}
Let $(\varphi,M,X)$ be semiflow on a compact Hausdorff space $X$. Then there always exists a point $p\in X$ at which $(\varphi,M,X)$ is multiply syndetically nonwandering.
\end{cor-1.2}

\begin{proof}
By Zorn's lemma, there exists a closed $\varphi$-invariant subset $X_0$ of $X$ such that the subsemiflow $(\varphi,M,X_0)$ is minimal. Then by Theorem~\ref{thm1.2}, every point $p\in X_0$ is multiply syndetically nonwandering for $(\varphi,M,X)$.
\end{proof}

Notice that in Theorem~\ref{thm1.2} if $(M,\pmb{+})$ is a discrete group and $g\in M$ not equal to $\e$, then $(\mathbb{Z},+)$ can be naturally imbedded in $(M,\pmb{+})$ via $n\mapsto ng$ as a subgroup. So by the classical theorem of Furstenberg and Weiss~\cite[Theorem~1.5]{FW}, we can obtain the multiple recurrence with respect to the element $g$. But importantly, $\{ng\,|\,n\in\mathbb{Z}\}$ is not a syndetic subgroup of $\{tg\,|\,t\in R\}$ in general.
Because of this reason, our Theorem~\ref{thm1.2} is not a consequence of Furstenberg and Weiss~\cite[Theorem~1.5]{FW} generally.

In 1978 Furstenberg and Weiss proved their regional multiple recurrence theorem by using the multiple Birkhoff recurrence theorem and homogeneity. However, that idea is not workable in our situation; this is because there is no applicable pointwise multiple recurrence theorem for commuting maps on a non-metrizable compact Hausdorff space and moreover the multiple returning time set of a multiply recurrent point is not syndetic in general.

Blaszczyk et al. in 1989~\cite{BPT} gave a topological proof of the topological multidimensional van der Waerden theorem by using induction and the associated inverse system.

However, Theorem~\ref{thm1.2} will be proved in $\S2$ following the nice idea of R.~Ellis by using his enveloping semigroup theory. Then based on Theorem~\ref{thm1.2} together with a dynamics concept---weak central set---introduced later, we can show Theorem~\ref{thm1.1} in $\S3$ using the idea of Bergelson, Furstenberg, Hindman and Katznelson 1989~\cite{BFK}. We will end $\S\ref{sec3}$ with a closely related open question for our further study.

\medskip
Finally the author is deeply grateful to Professor Hillel Furstenberg for his many enthusiastic helps and encouragements.

\section{The regional multiple recurrence theorem}\label{sec2}
In this section we will prove Theorem~\ref{thm1.2} stated in $\S\ref{sec1.2}$ following the framework of \cite{BFK}, also see \cite[$\S1.11$]{Gla}, completely different with \cite{FW, BPT}.

Henceforth let $(R,+,\cdot)$ be a semi-ring and $(M,\pmb{+})$
a $R$-semimodule defined as in $\S\ref{sec1.1}$; and let
$\varphi\colon M\times X\rightarrow X$, or denoted $(\varphi,M,X)$, be a semiflow on a compact Hausdorff space $X$ as in $\S\ref{sec1.2}$.
If we identify $(M,\pmb{+})$ with the transition semigroup $(\{\varphi(g,\centerdot)\}_{g\in M},\circ)$ via
$g\mapsto\varphi(g,\centerdot)$ and $g\pmb{+}h\mapsto\varphi^g\circ\varphi^h$,
then $X$ becomes a compact Hausdorff $M$-space.

\subsection{Subactions in topological dynamics}
In preparation for proving Theorem~\ref{thm1.2}, we need to introduce a result (Theorem~\ref{thm-p2.4} below which itself is of interest) on subactions in topological dynamics.

Let $\textrm{\large E}(G,X)$ be the Ellis enveloping semigroup of a semiflow $(G,X)$ with phase semigroup $G$ and with phase mapping $(g,x)\mapsto g(x)$, which is the closure of $G$ in $X^X$ provided with the topology of pointwise convergence (cf.~\cite{Ell60,E69,Fur,Aus}). An element $u\in\textrm{\large E}(G,X)$ is called a \textit{minimal idempotent} in $\textrm{\large E}(G,X)$ if and only if there is a minimal left ideal $I$ in $\textrm{\large E}(G,X)$ such that $u^2=u$ and $u\in I$.

Let $I$ be a minimal left ideal in $\textrm{\large E}(G,X)$. Then $x\in X$ is an a.p. point of $(G,X)$ iff there is an idempotent $u\in I$ with $ux=x$; moreover, $Ix$ is a $G$-minimal subset of $X$ for all $x\in X$. See, e.g., \cite{Ell60,E69,Aus}.

\begin{lem}\label{lem-p1}
Let $(G,X)$ be a semiflow, $S$ a subsemigroup of $G$ such that $\textrm{\large E}(G,X)=\textrm{\large E}(S,X)$, and $x\in X$. Then $x$ is an a.p. point of $(S,X)$ if and only if $x$ is an a.p. point of $(G,X)$. In this case $\overline{Sx}=\overline{Gx}$.
\end{lem}

\begin{proof}
Let $\mathbb{I}$ be a minimal left ideal in $\textrm{\large E}(G,X)$. Then $\mathbb{I}$ is also a minimal left ideal in $\textrm{\large E}(S,X)$. It is a well-known fact that $x$ is a.p. if and only if $x\in\mathbb{I}x$ and moreover $\mathbb{I}x$ is exactly the orbit-closure of $x$ when $x$ is an a.p. point. Hence Lemma~\ref{lem-p1} holds.
\end{proof}

In our later applications, it is hard to verify the condition $\textrm{\large E}(G,X)=\textrm{\large E}(S,X)$ required by Lemma~\ref{lem-p1}. In fact, from the above proof, we can easily obtain the following weaker, but convenient, result.

\begin{lem}\label{lem-p2}
Let $(G,X)$ be a semiflow and $S$ a subsemigroup of $G$. If there is a minimal left ideal $\mathbb{I}$ in $\textrm{\large E}(S,X)$ such that each idempotent in $\mathbb{I}$ is a minimal idempotent in $\textrm{\large E}(G,X)$, then every a.p. point of $(S,X)$ is an a.p. point of $(G,X)$.
\end{lem}

\begin{proof}
Let $x$ be an a.p. point of $(S,X)$. Then there is an idempotent $u\in\mathbb{I}$ with $ux=x$. Since $u$ is also a minimal idempotent in $\textrm{\large E}(G,X)$, $x$ is an a.p. point of $(G,X)$.
\end{proof}

We notice that $S\subset G$ implies that $\textrm{\large E}(S,X)\subseteq\textrm{\large E}(G,X)$. However, since $S$ need not be a left ideal in $G$, hence $\textrm{\large E}(S,X)$ is not necessarily an invariant subset of $\textrm{\large E}(G,X)$ so that a minimal left ideal in $\textrm{\large E}(S,X)$ need not be a minimal left ideal in $\textrm{\large E}(G,X)$.

Next we can provide with a sufficient condition for the hypothesis required by Lemma~\ref{lem-p2}, whose proof is mainly motivated by an argument due to R.~Ellis.

\begin{lem}\label{lem-p2.3}
Let $G$ be a semigroup and $S$ a subsemigroup of $G$. Assume $\{(G,X_i)\,|\,i\in\Theta\}$ is a family of semiflows such that $\textrm{\large E}(G,X_i)=\textrm{\large E}(S,X_i)$ for all $i\in\Theta$, and $X=\prod_{i\in\Theta}X_i$. Then every minimal left ideal $I$ in $\textrm{\large E}(S,X)$ is such that if $u\in I$ is an idempotent $u$ is also a minimal idempotent in $\textrm{\large E}(G,X)$.
\end{lem}

\begin{proof}
Let $\pi_i\colon X\rightarrow X_i$, $\pi_{i,*}\colon \textrm{\large E}(G,X)\rightarrow\textrm{\large E}(G,X_i)$, and $\pi_{i,*}\colon\textrm{\large E}(S,X)\rightarrow\textrm{\large E}(S,X_i)$ all be the canonical projections, for all $i\in\Theta$.
Let $I$ be a minimal left ideal in $\textrm{\large E}(S,X)$ with an idempotent $u\in I$. Then $R_u\colon p\mapsto pu$ is continuous from $\textrm{\large E}(G,X)$ to itself.

Since $\textrm{\large E}(G,X)u$ is a closed left ideal in $\textrm{\large E}(G,X)$, we can choose by Zorn's lemma a minimal left ideal $K$ in $\textrm{\large E}(G,X)$ with $K\subseteq\textrm{\large E}(G,X)u$. Let $v$ be an idempotent in $K$. Clearly $vu=v$, this is because $v=fu$ for some $f\in\textrm{\large E}(G,X)$ and so $vu=fu^2=fu=v$. Next we proceed to showing that $uv=u$.

Indeed, for each $i\in\Theta$, we simply set $u_i=\pi_{i,*}(u)\in\textrm{\large E}(G,X_i)$ and $v_i=\pi_{i,*}(v)\in\textrm{\large E}(S,X_i)$. As $u=(u_i\,|\,i\in\Theta)$ and $v=(v_i\,|\,i\in\Theta)$ (cf., e.g.,~\cite[Proposition~3.9]{E69}), to prove $uv=u$, it suffices to show that $u_iv_i=u_i$ for all $i\in\Theta$.

Let $I_i=\pi_{i,*}(I)$ and $K_i=\pi_{i,*}(K)$, which both are minimal left ideals in $\textrm{\large E}(G,X_i)=\textrm{\large E}(S,X_i)$ such that $u_i\in I_i$ and $v_i\in K_i$, for all $i\in\Theta$. Since $\pi_{i,*}$ is a semigroup homomorphism, hence $v_iu_i=v_i, K_iv_i=K_i$ and $I_iu_i=I_i$. In particular, by $v_iI_i\subseteq I_i$ for $v_iI_i\subseteq K_iI_i\subseteq\textrm{\large E}(G,X_i)u_iI_i=\textrm{\large E}(S,X_i)u_iI_i\subseteq I_i$, it follows that $v_i=v_iu_i\in I_i$ and then $v_i\in I_i\cap K_i$ so that $I_i=K_i$. This implies that $u_iv_i=u_i$ for all $i\in\Theta$ and thus $uv=u$ as desired.

Therefore by $uv\in K$, we see $u\in K$. So $u$ is a minimal idempotent in $\textrm{\large E}(G,X)$. This completes the proof of Lemma~\ref{lem-p2.3}.
\end{proof}

It should be mentioned that condition $\textrm{\large E}(G,X_i)=\textrm{\large E}(S,X_i)$ for all $i\in\Theta$ need not imply that $\textrm{\large E}(G,X)=\textrm{\large E}(S,X)$ in the foregoing Lemma~\ref{lem-p2.3}.

\begin{thm}\label{thm-p2.4}
Let $G$ be a semigroup and $S$ a subsemigroup of $G$. Suppose $\{(G,X_i)\,|\,i\in\Theta\}$ is a family of semiflows with $\textrm{\large E}(G,X_i)=\textrm{\large E}(S,X_i)$ for all $i\in\Theta$ and $X=\prod_{i\in\Theta}X_i$. Then every a.p. point of $(S,X)$ is an a.p. point of $(G,X)$.
\end{thm}

\begin{proof}
By Lemma~\ref{lem-p2.3}, it follows that every minimal left ideal $I$ in $\textrm{\large E}(S,X)$ is such that if $u\in I$ is an idempotent $u$ is also a minimal idempotent in $\textrm{\large E}(G,X)$. Then Lemma~\ref{lem-p2}, every a.p. point of $(S,X)$ is an a.p. point of $(G,X)$.
\end{proof}

Let $(\Gamma,X)$ be a flow with phase group $\Gamma$ and let $\widetilde{\Gamma}(X)\subset X^X$ be the group of homeomorphisms of $X$ onto itself defined by $\Gamma$. Then by using almost periodic sets and induction, the special case that $\widetilde{G}(X_i)=\widetilde{S}(X_i)$ for each $i\in\Theta$ has been proved by Auslander and Akin in~\cite[Theorem~3]{AA}.

However, our condition $\textrm{\large E}(G,X_i)=\textrm{\large E}(S,X_i)$ is weaker than $\widetilde{G}(X_i)=\widetilde{S}(X_i)$. For example, let $\pi\colon\mathbb{R}\times X\rightarrow X$ be an effective $C^0$-flow, $G=\mathbb{R}$ and $S=\mathbb{Q}$; then $\textrm{\large E}(G,X)=\textrm{\large E}(S,X)$ but $\mathbb{R}\approx\widetilde{G}(X)\supsetneqq\widetilde{S}(X)\approx\mathbb{Q}$.

Let $I$ be a non-empty index set. The following theorem, as a corollary of Theorem~\ref{thm-p2.4}, is a generalization of a theorem of Glasner~\cite{G}.

\begin{thm}\label{thm2.5}
Let $\varphi\colon \Gamma\times X\rightarrow X$ be a flow with phase group $\Gamma$ and $\Psi=\langle T_i\,|\,i\in I\rangle$ be the subgroup of $\Gamma$ generated by elements $T_i\in \Gamma$ for $i\in I$. If $\hat{x}\in X^I$ is an a.p. point of $\Psi$, then $\hat{x}$ is also an a.p. point of the group $G$ generated by $\hat{T}$ and $\{\hat{g}\,|\,g\in\Psi\}$, where $\hat{T}\colon(x_i)\mapsto(T_ix_i)$ and $\hat{g}\colon (x_i)\mapsto(gx_i)$ are of $X^I$ onto $X^I$.
\end{thm}

\begin{proof}
Let $X_i=X$ for all $i\in I$ and $S=\{\hat{g}\,|\,g\in\Psi\}$. Let $(G,X_i)$ be naturally induced by $(G,X^I)$. Then $S$ is a subgroup of $G$ such that
$\textrm{\large E}(G,X_i)=\textrm{\large E}(S,X_i)$ for all $i\in I$. Whence the statement follows from Theorem~\ref{thm-p2.4}.
\end{proof}

Let $f\colon X\rightarrow X$ be a homeomorphism. If $\Gamma=\{f^i\,|\,i\in\mathbb{Z}\}$ and $T_1=f, T_2=f^2, \dotsc, T_n=f^n$ where $1\le n<\infty$, then the statement of Theorem~\ref{thm2.5} is due to Glasner (also see \cite[Theorem~4]{AA}).

Clearly the foregoing Theorem~\ref{thm2.5} can be stated and proved in the case that $\Gamma$ and $\Psi$ are semigroups.

\subsection{A technical lemma}
Given any index set $\mathcal{I}$, any $T_i\in M$ for $i\in\mathcal{I}$, by $\langle T_i\,|\,i\in\mathcal{I}\rangle_R$ we denote, from here on, the sub-semimodule of $M$, which is generated by $\{T_i\,|\,i\in\mathcal{I}\}$, over $(R,+,\cdot)$.
Next we define, for all $g\in M$, the continuous transformation $\hat{g}$ of $X^\mathcal{I}$ to itself by
\begin{gather*}
\hat{g}\colon X^\mathcal{I}\rightarrow X^\mathcal{I};\quad \hat{x}=(x_i)_{i\in\mathcal{I}}\mapsto\hat{g}(\hat{x})=(g(x_i))_{i\in\mathcal{I}},
\end{gather*}
noting here that we have identified $g$ with $\varphi^g$, i.e. $g(x)=\varphi(g,x)\ \forall x\in X$, based on $(\varphi,M,X)$.
On the other hand, for $\{T_i\,|\,i\in\mathcal{I}\}$, we can define for all $t\in R$
\begin{gather*}
\hat{T}^t\colon X^\mathcal{I}\rightarrow X^\mathcal{I};\quad \hat{x}=(x_i)_{i\in\mathcal{I}}\mapsto\hat{T}^t(\hat{x})=(T_i^{t}(x_i))_{i\in\mathcal{I}},
\end{gather*}
where we have identified $T_i^t$ with $\varphi(tT_i,\centerdot)$ for all $i\in\mathcal{I}$ based on $(\varphi,M,X)$.

It is a well-known fact that for any minimal subset $\Lambda$ of $(\varphi,M,X)$, usually $\Lambda$ need not be minimal for $(\varphi,\langle T_i\,|\,i\in\mathcal{I}\rangle_R,X)$; even nor for the classical case $(R,+,\cdot)=(\mathbb{Z},+,\cdot)$ like $\Lambda$ consists of a periodic orbit.
Because of this reason the following is an important technical lemma for proving Theorem~\ref{thm1.2} later, whose special case that $(M,\pmb{+})=(\mathbb{Z},+), \mathcal{I}=\{1,2,\dotsc,n\}$ and $R=(\mathbb{Z},+,\cdot)$ is~\cite[Proposition~1.55]{Gla}).

\begin{lem}\label{lem2.6}
Given any $T_i\in M$ for $i\in\mathcal{I}$, let $(\varphi,\langle T_i\,|\,i\in\mathcal{I}\rangle_R,X)$ be the semiflow induced by $(\varphi,M,X)$. Define two semiflows on $X^\mathcal{I}$
\begin{gather*}
\hat{T}\colon R\times X^\mathcal{I}\rightarrow X^\mathcal{I}\textrm{ denoted }(\hat{T},R,X^\mathcal{I});\quad (t,\hat{x})\mapsto\hat{T}^t(\hat{x})\\
\hat{\varphi}\colon{\langle T_i\,|\,i\in\mathcal{I}\rangle}_R\times X^\mathcal{I}\rightarrow X^\mathcal{I}\textrm{ denoted }{(\hat{\varphi},\langle T_i\,|\,i\in\mathcal{I}\rangle}_R,X^\mathcal{I});\quad (g,\hat{x})\mapsto\hat{g}(\hat{x})
\intertext{and let}
\mathfrak{T}=\left\{\hat{T}^t\hat{g}\colon X^\mathcal{I}\rightarrow X^\mathcal{I}\,|\,(t,g)\in R\times{\langle T_i\,|\,i\in\mathcal{I}\rangle}_R\right\}.
\end{gather*}
If $\varLambda$ is a $\hat{\varphi}$-minimal subset of $X^\mathcal{I}$ and let
$\varSigma=\textrm{cls}_{X^\mathcal{I}}^{}{\bigcup}_{t\in R}\hat{T}^t[\varLambda]=\textrm{cls}_{X^\mathcal{I}}^{}\hat{T}(R\times\varLambda)$,
then $\varSigma$ is a $\mathfrak{T}$-minimal subset of $X^\mathcal{I}$.
\end{lem}

\begin{proof}
Since for all $i\in\mathcal{I}$ we have $T_i^tg=\varphi(tT_i\pmb{+}g,\centerdot)=gT_i^t\colon X\rightarrow X$ for all $g\in M, t\in R$ and so
\begin{gather*}
(\hat{T}^t\hat{g})(\hat{T}^{t^\prime}{\hat{g^\prime}})=(\hat{T}^{t^\prime}{\hat{g^\prime}})(\hat{T}^t\hat{g})=\hat{T}^{t+t^\prime}\widehat{(g^\prime\pmb{+}g)}\in\mathfrak{T}\quad \forall (t,g), (t^\prime, g^\prime)\in R\times\langle T_i\,|\,i\in\mathcal{I}\rangle_R,
\end{gather*}
then $\mathfrak{T}$ is an abelian multiplicative semigroup of continuous transformations of $X^\mathcal{I}$ and $\varSigma$ is $\mathfrak{T}$-invariant.

Let $G=\mathfrak{T}$ and $S=\{\hat{g}\,|\,g\in{\langle T_i\,|\,i\in\mathcal{I}\rangle}_R\}$.
Then $S$ may be regarded as a subsemigroup of $G$. Define $(G,X_i)$, where $X_i=X$ for all $i\in\mathcal{I}$, by the ways of
$$
\hat{T}^t\hat{g}(x_i)=T_i^tg(x_i)\quad \forall x_i\in X_i.
$$
Because $RT_i\pmb{+}\langle T_i\,|\,i\in\mathcal{I}\rangle_R=\langle T_i\,|\,i\in\mathcal{I}\rangle_R$,
$\textrm{\large E}(G,X_i)=\textrm{\large E}(S,X_i)$ for all $i\in\mathcal{I}$.

Then by Theorem~\ref{thm-p2.4}, it follows that for all $\hat{x}\in\varLambda$, $\overline{G[\hat{x}]}=\varSigma$ is a minimal subset of $(\mathfrak{T},X^\mathcal{I})$.
The proof of Lemma~\ref{lem2.6} is therefore completed.
\end{proof}

We notice that although the statement of Lemma~\ref{lem2.6} still holds without the commutativity of $M$ as Theorem~\ref{thm2.5}, yet the commutativity is for guaranteeing that $\mathfrak{T}$ is of the form $R\times\langle T_i\,|\,i\in\mathcal{I}\rangle_R$. This is an important point in the proof of Theorem~\ref{thm1.2}.

In addition, if $(M,\pmb{+})$ is a right $R$-semimodule, then we need to define $\hat{T}$ in Lemma~\ref{lem2.6} via $T_i^t=\varphi(T_it,\centerdot)$ in place of $T_i^t=\varphi(tT_i,\centerdot)$.

\subsection{Topological multidimensional van der Waerden Theorem}
With Lemma~\ref{lem2.6} at hands, now we can readily prove our regional multiple recurrence theorem for any semi-modules acting on a compact Hausdorff space $X$.

This kind of result of Theorem~\ref{thm1.2} is also called Topological Multidimensional van der Waerden Theorem in the literature; see, e.g., \cite{BKW, BPT}.

\begin{proof}[\textbf{Proof of Theorem~\ref{thm1.2}}]
Let $(\varphi,M,X)$ be a minimal semiflow on the compact Hausdorff space $X$ and let $T_1,\dotsc,T_l\in M$ be any given. Let $U$ be an arbitrary open subset of $X$ with $U\not=\emptyset$.

Let $\mathcal{I}=M$ and $T_i=i$ for all $i\in M$. Then $\langle T_i\,|\,i\in M\rangle_R=M$ so that $(\varphi,\langle T_i\,|\,i\in M\rangle_R,X)$ is minimal.
Let $\hat{T}\colon R\times X^M\rightarrow X^M,\ \hat{\varphi}\colon M\times X^M\rightarrow X^M$ and $\mathfrak{T}=\left\{\hat{T}^t\hat{g}\colon X^M\rightarrow X^M\,|\,(t,g)\in R\times M\right\}$ all be defined as in Lemma~\ref{lem2.6}; that is,
$$
\hat{T}\colon(t,\hat{x})\mapsto\hat{T}^t(\hat{x})=\left(\varphi(ti,x_i)\right)_{i\in M}\quad \textrm{and}\quad \hat{\varphi}\colon(g,\hat{x})\mapsto\hat{g}(\hat{x})=\left(\varphi(g,x_i)\right)_{i\in M}\quad \forall \hat{x}=(x_i)_{i\in M}\in X^M.
$$
Now set
\begin{gather*}
\varLambda=\varDelta_{M}(X)=\left\{(x_i)_{i\in M}\,|\,x_i\equiv x\in X\textrm{ for all }i\in M\right\}\subset X^M\quad{\textrm{and}}\quad \varSigma=\mathrm{cls}_{X^M}^{}{\bigcup}_{t\in R}\hat{T}^t[\varDelta_M(X)].
\end{gather*}
Since $(\varphi,\langle T_i\,|\,i\in M\rangle_R,X)=(\varphi,M,X)$ is minimal by hypothesis, then $\varDelta_M(X)$ is $\hat{\varphi}$-minimal in $X^M$. Hence by Lemma~\ref{lem2.6}, any point
$\hat{x}=(x)_{i\in M}\in\varDelta_M(X)$ with $x\in U$
is an a.p. point of
$$\mathfrak{T}\colon R\times M\times \varSigma\rightarrow \varSigma,$$
that is defined by $((t,g),\hat{z})\mapsto\hat{T}^t\hat{g}(\hat{z})$ for all $(t,g)\in R\times M$ and $\hat{z}\in\varSigma$. Hence for any $i_1,\dotsc, i_l\in M$, the returning-time set
\begin{gather*}
N_{\mathfrak{T}}\left(\hat{x},\widehat{U}_{i_1,\dotsc,i_l}\right)=\left\{(t,g)\in R\times M\,\big{|}\,\hat{T}^t\hat{g}(\hat{x})\in\widehat{U}_{i_1,\dotsc,i_l}\right\},
\end{gather*}
where $\widehat{U}_{i_1,\dotsc,i_l}=\prod_{i\in\mathcal{I}}U_i$ such that $U_{i_k}=U$ for $1\le k\le l$ and $U_i=X$ for $i\not=i_k$ and where $R\times M$ is thought of as a discrete additive semigroup, is syndetic in $(R\times M,+)$; that is, one can find a finite subset $K=\{(s_1,g_1),\dotsc,(s_k,g_k)\}$ of $R\times M$ such that
\begin{gather*}
(K+\tau)\cap N_{\mathfrak{T}}\left(\hat{x},\widehat{U}_{i_1,\dotsc,i_l}\right)\not=\emptyset\quad \forall \tau\in R\times M.
\end{gather*}
Thus for all $(t,g)\in N_{\mathfrak{T}}\left(\hat{x},\widehat{U}_{i_1,\dotsc,i_l}\right)$, we have
$\varphi(ti_k,\varphi(g,x))\in U$ for $k=1,\dotsc,l$.
Therefore,
\begin{gather*}
\varphi(g,x)\in\bigcap_{k=1}^l\varphi^{-ti_k}[U]\not=\emptyset\quad \forall\, (t,g)\in N_{\mathfrak{T}}\left(\hat{x},\widehat{U}_{i_1,\dotsc,i_l}\right)
\end{gather*}
where $\varphi^{-ti_k}={\varphi(ti_k,\centerdot)}^{-1}$. Let
\begin{gather*}
S={\Pr}_1\left[N_{\mathfrak{T}}\left(\hat{x},\widehat{U}_{i_1,\dotsc,i_l}\right)\right],\quad \textrm{where }{\Pr}_1\colon (t,g)\mapsto t\textrm{ of }R\times M\textrm{ onto }R.
\end{gather*}
Then $S\subseteq N_{T_1,\dotsc,T_l}(U)$ and $S$ is syndetic in $(R,+)$ for $i_1=T_1, \dotsc,i_l=T_l\in \mathcal{I}=M$.

This thus completes the proof of Theorem~\ref{thm1.2}.
\end{proof}

Theorem~\ref{thm1.2} implies the classical multiple Birkhoff recurrence theorem of Furstenberg and Weiss~(\cite[Theorem~1.4]{FW} or \cite[Theorem~2.6]{Fur}), if $X$ is a compact metric space.

\begin{FW}[Furstenberg-Weiss 1978; A bi-sided version]
Let $T_1,\dotsc,T_p$ be commuting homeomorphisms of a compact metric space $X$. Then there is a residual subset $\Sigma$ of any $\{T_1,\dotsc,T_p\}$-minimal subset of $X$ such that for all $x\in\Sigma$ there is a sequence $n_k\to\infty$ with
$T_i^{n_k}x\to x$ and $T_i^{-n_k}x\to x$
simultaneously for $i=1,\dotsc,p$.
\end{FW}

\begin{proof}
Let $M$ be the abelian group generated by $\{T_1,\dotsc,T_p\}$, which is thought of as a module over the ring $(\mathbb{Z},+,\cdot)$. Let $T_{p+1}=T_1^{-1}, \dotsc,T_{2p}=T_p^{-1}$. By restricting to a subset of $X$ if necessary, we can suppose without loss of generality that $(\varphi,M,X)$ is a minimal flow with phase map $\varphi\colon(g,x)\mapsto g(x)$.

Now for any $i\ge1$, let $\{U_{i,k}\,|\,k=1,2,\dotsc\}$ be an open cover of $X$ of diameter less than $1/i$. Then by Theorem~\ref{thm1.2}, it follows that for any $i\ge1$ and $m\ge1$,
$$
W_{i,m}:=\bigcup_{k=1}^\infty\bigcup_{n=m}^\infty\left(U_{i,k}\cap T_1^nU_{i,k}\cap\dotsm\cap T_p^nU_{i,k}\cap T_{p+1}^nU_{i,k}\cap\dotsm\cap T_{2p}^nU_{i,k}\right)
$$
is an open dense subset of $X$. Let $\Sigma=\bigcap_{i=1}^\infty\bigcap_{m=1}^\infty W_{i,m}$, which is a residual set in $X$. Clearly each point of $\Sigma$ satisfies the requirement of the recurrence theorem.
\end{proof}

\begin{que}\label{que2.6}
If $(R,+,\cdot)=(\mathbb{Z}_+,+,\cdot)$, then $N_{T_1,\dotsc,T_l}(U)$ in Theorem~\ref{thm1.2} should be an $\textrm{IP}^*$-set; in other words, $N_{T_1,\dotsc,T_l}(U)$ should intersect every IP-set in $(\mathbb{Z}_+,+)$.
\end{que}

We note that if $X$ is a compact metric space, then by \cite[Theorem~2.16]{Fur} and homogeneity we can see that $N_{T_1,\dotsc,T_l}(U)$ is an IP$^*$-set in $\mathbb{Z}_+$. However, in the present non-metrizable situation, there does not need to exist a multiply recurrent point of $(T_1,\dotsc,T_l)$ (cf.~\cite{BK} for counterexamples).

\subsection{Birkhoff recurrent sets}
Let $X$ be a compact Hausdorff space and let $(R,+,\cdot)$ be a semi-ring with $R\not=\{0\}$. As before, by a semiflow $\varphi\colon R\times X\rightarrow X$ or denoted $(\varphi,R,X)$, we mean that $\varphi(t,x)$ is such that:
\begin{itemize}
\item $\varphi(0,x)=x\ \forall x\in X$, i.e., $\varphi(0,\centerdot)=i_X$;
\item $\varphi^t=\varphi(t,\centerdot)\colon X\rightarrow X$, for any $t\in R$, is continuous; and

\item $\varphi(s+t,x)=\varphi(s,\varphi(t,x))$, i.e., $\varphi^{s+t}=\varphi^s\circ\varphi^t\; \forall s,t\in R$.
\end{itemize}
Note here that $(R,+)$ itself is an $R$-semimodule with $\e=0$.

\begin{defn}\label{def2.8}
Given any nonzero elements $t_1,\dotsc,t_l\in R$, a subset $\Gamma$ of $R$ with $0\not\in\Gamma$ is call a \textit{set of $(t_1,\dotsc,t_l)$-recurrence} if for each minimal $(\varphi,R,X)$ and any open set $U\subseteq X, U\not=\emptyset$ we have
\begin{gather*}
\Gamma\cap N_{\varphi;t_1,\dotsc,t_l}(U)\not=\emptyset,\quad \textrm{where }N_{\varphi;t_1,\dotsc,t_l}(U)=\left\{t\in R\,|\,U\cap\varphi^{-tt_1}[U]\cap\dotsm\cap\varphi^{-tt_l}[U]\not=\emptyset\right\}.
\end{gather*}
See, e.g.,~\cite[Definition~2.6]{HKM} and \cite{HSY} for the special case that $R=\mathbb{Z}_+$ and $t_1=1,\dotsc,t_l=l$ but requiring $X$ a compact metric space there.
\end{defn}

Then we can easily obtain the following characterizations, in which the commutativity of $(R,+)$ plays a role.

\begin{thm}\label{thm2.9}
Let $\Gamma\subset R$ and $F=\{0,t_1,\dotsc,t_l\}\subseteq R$ be any given with $t_i\not=0$ for $1\le i\le l$ and $0\not\in\Gamma$; then the following statements are pairwise equivalent.
\begin{enumerate}
\item[$(1)$] $\Gamma$ is a set of $(t_1,\dotsc,t_l)$-recurrence.
\item[$(2)$] Given any $(\varphi,R,X)$ and open cover $\mathfrak{U}=\{U_1,\dotsc,U_q\}$ of $X$, there is some $1\le j\le q$ with
    $\Gamma\cap N_{\varphi;t_1,\dotsc,t_l}(U_j)\not=\emptyset$.

\item[$(3)$] Given any partition $R=B_1\cup\dotsm\cup B_q$, one of these cells $B_j$ has the property that there exist $a\in R$ and $d\in\Gamma$ with $a+dF\subseteq B_j$.
\item[$(4)$] Given any syndetic set $E\subseteq (R,+)$, there are $a\in R$ and $d\in\Gamma$ with $a+dF\subseteq E$.
\end{enumerate}
\end{thm}

\begin{proof}
$(1)\Leftrightarrow(2)$ is obvious from Definition~\ref{def2.8} and Zorn's lemma. $(3)\Rightarrow(2)$ follows from the proof of Corollary~1.1.\ref{cor1.1-3}. Indeed, let $x\in X$ and define $B_j=\{t\,|\,\varphi(t,x)\in U_j\}$ for $j=1,\dotsc,q$. Then by (3), there is a $B_j$ such that $a+dF\subseteq B_j$ for some $a\in R$ and $d\in\Gamma$. That is to say, $\varphi(a,x)\in U_j$ and $\varphi(a+dt_i,x)\in U_j$ for $1\le i\le l$. Thus $d\in\Gamma\cap N_{\varphi;t_1,\dotsc,t_l}(U_j)\not=\emptyset$.

$(4)\Rightarrow(1)$ is trivial, since each point is a.p. for any minimal $(\varphi,R,X)$. Indeed, let $(\varphi,R,X)$ be minimal, $U$ an open nonempty subset of $X$, and $x\in U$. Then $E=\{t\,|\,\varphi(t,x)\in U\}$ is syndetic so that there are $a\in R$ and $d\in\Gamma$ such that $a+dF\subseteq E$. Thus $\varphi(a,x)\in U$ and $\varphi(a+dt_i,x)\in U$ for $1\le i\le l$. Whence $d\in\Gamma\cap N_{\varphi;t_1,\dotsc,t_l}(U)\not=\emptyset$.

$(2)\Rightarrow(3)$: By Note~\ref{thm1.2}.\ref{n1.2.6}, it follows from (2) that for any minimal $(\varphi,R,X)$ and nonempty open subset $U$ of $X$, $N_{\varphi;t_1,\dotsc,t_l}(U)$ is syndetic in $(R,+)$. Moreover, by Lemma~\ref{lem3.4}, some $B_j$ is a ``weak central set'' of $(R,+)$. Then by Lemma~\ref{lem3.3} and Footnote~5 there, we can find $a\in R$ and $d\in\Gamma$ such that $a+dF\subseteq B_j$.

$(1)\Rightarrow(4)$ follows from Theorem~\ref{thm3.6}; see Footnote~7 there. This proves Theorem~\ref{thm2.9}.
\end{proof}

This theorem is a generalization of \cite[Theorem~2.5]{HKM}. However, since in our situation $R$ is possibly uncountable, the standard topological version of Fustenberg's Correspondence Principle will need to be adapted for non-metrizable compact Hausdorff space.

Recall that a property is \textit{Ramsey} if for any set $\Gamma\subseteq R$ having this property and any partition $\Gamma=A\cup B$, at least one of $A$ and $B$ has this property (cf.~\cite{Fur}).

The following proposition is a direct corollary of Theorem~\ref{thm2.9}, which is a generalization of the special case $\mathbb{Z}$ (\cite[Proposition~7.2.4]{HSY} or \cite[Proposition~6.2]{HKM}).

\begin{prop}\label{prop2.10}
Given any nonzero elements $t_1,\dotsc,t_l\in R$, the family of subsets of $(t_1,\dotsc,t_l)$-recurrence of $R$ has the Ramsey property.
\end{prop}

Further, as a result of the foregoing proposition, we can easily get the following that generalizes \cite[Corollary~6.3]{HKM}.

\begin{prop}\label{prop2.11}
Let $t_1,\dotsc,t_l\in R$ be nonzero and let $\Gamma\subset R$ be a set of $(t_1,\dotsc,t_l)$-recurrence. Let $(\varphi,R,X)$ be minimal and $U$ a nonempty open subset of $X$. Then $\Gamma\cap N_{\varphi;t_1,\dotsc,t_l}(U)$ is a set of $(t_1,\dotsc,t_l)$-recurrence.
\end{prop}

Definition~\ref{def2.8} may be strengthened as follows:

\begin{defn}\label{def2.12}
Given any integer $l\ge1$, a subset $\Gamma$ of $R$ with $0\not\in\Gamma$ is call a \textit{set of $l$-recurrence for $R$} if for each minimal $(\varphi,R,X)$, any nonzero elements $t_1,\dotsc,t_l\in R$ and any nonempty open set $U\subseteq X$ we have
$\Gamma\cap N_{\varphi;t_1,\dotsc,t_l}(U)\not=\emptyset$. Further, $\Gamma$ is called a \textit{set of multiple recurrence for $R$} if it is a set of $l$-recurrence for each $l\ge1$.
\end{defn}

Then one can obtain characterizations similar to Theorem~\ref{thm2.9}. Moreover, the multiple recurrence has the Ramsey property:

\begin{prop}\label{prop2.13}
The family of subsets of multiple recurrence for $R$ has the Ramsey property.
\end{prop}

\begin{proof}
Suppose the contrary that $\Gamma=\Gamma_1\cup\Gamma_2$ is a set of multiple recurrence for $R$ such that $\Gamma_1$ and $\Gamma_2$ both are not. Then by $(1)\Leftrightarrow(2)$ of Theorem~\ref{thm2.9}, there are two minimal semiflows, say $(\varphi,R,X)$ and $(\psi,R,Y)$, nonzero elements $t_1,\dotsc,t_l, \tau_1,\dotsc,\tau_m\in R$, and open covers $\mathfrak{U}=\{U_1,\dotsc,U_q\}$ of $X$ and $\mathfrak{V}=\{V_1,\dotsc,V_p\}$ of $Y$ such that
\begin{gather*}
\Gamma_1\cap N_{\varphi;t_1,\dotsc,t_l}(U_i)=\emptyset\quad 1\le i\le q\intertext{and} \Gamma_2\cap N_{\psi;\tau_1,\dotsc,\tau_m}(V_j)=\emptyset\quad 1\le j\le p.
\end{gather*}
Now we consider the naturally induced system
$$
\varphi\times\psi\colon R\times X\times Y\rightarrow X\times Y; \quad (t,(x,y))\mapsto(\varphi(t,x),\psi(t,y)),
$$
and the open cover $\mathfrak{U}\times\mathfrak{V}=\{U_i\times V_j; 1\le i\le q, 1\le j\le p\}$ of $X\times Y$. Since $\Gamma$ is multiple recurrent for $R$, then
$$
d\in\Gamma\cap N_{\varphi\times\psi; t_1,\dotsc,t_l, \tau_1,\dotsc,\tau_m}(U_i\times V_j)\not=\emptyset
$$
for some pair $(i,j)$. This implies that
\begin{gather*}
U_i\cap\varphi^{-dt_1}[U_i]\cap\dotsm\cap\varphi^{-dt_l}[U_i]\not=\emptyset\intertext{and} V_j\cap\psi^{-d\tau_1}[V_j]\cap\dotsm\cap\psi^{-d\tau_m}[V_j]\not=\emptyset.
\end{gather*}
Then either $\Gamma_1\cap N_{\varphi;t_1,\dotsc,t_l}(U_i)\not=\emptyset$ or $\Gamma_2\cap N_{\psi;\tau_1,\dotsc,\tau_m}(V_j)\not=\emptyset$. This contradiction concludes the proof of Proposition~\ref{prop2.13}.
\end{proof}

However, the Ramsey property of sets of $l$-recurrence for $R$ keeps open.

\section{Weak central sets and van der Waerden theorem}\label{sec3}
This section will be mainly devoted to proving the van der Waerden theorem (Theorem~\ref{thm1.1}) based on the Regional Multiple Recurrence Theorem (Theorem~\ref{thm1.2}) and Weak Central Sets of discrete semigroups introduced below.

Moveover we will consider van der Waerden subset of any semi-module in this section; see Definition~\ref{def3.5} below.
\subsection{Weak central sets of semigroups}
The following concept is a slight generalization of Furstenberg's central sets of the natural number semigroup $(\mathbb{N},+)$ (cf.~\cite[Definition~8.3]{Fur}).

\begin{defn}\label{def3.1}
A subset $S$ of a discrete additive semigroup $(G,\pmb{+})$ is referred to as a \textit{weak central set} of $(G,\pmb{+})$, where $G$ is not necessarily abelian, if
\begin{itemize}
\item there exists a semiflow $\varphi\colon G\times X\rightarrow X$ on a compact Hausdorff space $X$,
\item a point $x\in X$ and an a.p. point $y\in X$ of $(\varphi,G,X)$ that is \textit{weakly proximal to $x$} under $\varphi$ in the sense that $\mathrm{cls}_X^{}G_\varphi[x]\cap\mathrm{cls}_X^{}G_\varphi[y]\not=\emptyset$,
\item and there is an open neighborhood $U$ of $y$,
\end{itemize}
such that $S=N_\varphi(x,U):=\left\{g\in G\,|\,\varphi(g,x)\in U\right\}$.
\end{defn}

It should be noticed here that comparing our orbital proximality with Furstenberg's central sets, $(x,y)$ is not necessarily to be a classical proximal pair for $(\varphi,G,X)$ (in the sense that there exists a net $t_\theta\in G$ with $\varphi(t_\theta,x)\to z, \varphi(t_\theta,y)\to z$ for some $z\in X$) and $X$ is not necessarily metrizable in our Definition~\ref{def3.1}.

Clearly, $G$ is itself a weak central set of $(G,\pmb{+})$ by considering a semiflow on a singleton set. If $x$ is itself a.p. for $(\varphi,G,X)$, then $N_\varphi(x,U)$ is a syndetic weak central set of $(G,\pmb{+})$. Of course, a weak central set need not be syndetic in general.

Let $(G,\pmb{+})$ be a semigroup not necessarily abelian. Recall that a subset $S$ of $G$ is called \textit{piecewise syndetic} if
one can find a syndetic subset $S^\prime$ of $G$ such that for any finite subset $A$ of $S^\prime$, there is some $g_A^{}\in S$ with $A\pmb{+}g_A^{}\subset S$.
We can easily check that if $S\subseteq G$ is piecewise syndetic in $G$, then there exists a finite set $F\subseteq G$ such that for any finite set $A\subseteq G$, it holds that $\bigcup_{f\in F}L_f^{-1}S\supseteq A\pmb{+}g_A^{}$ for some $g_A^{}\in S$. Here $L_f\colon G\rightarrow G$ by $g\mapsto f\pmb{+}g$, for each $f\in G$.

Next we will present some basic combinatorial properties of weak central sets, which are generalizations of properties of Furstenberg central sets (cf.,~e.g.,~\cite[Proposition~8.9]{Fur}).

\begin{lem}\label{lem3.2}
Let $S$ be a weak central set of a semigroup $(G,\pmb{+})$ not necessarily abelian. Then it is piecewise syndetic in $(G,\pmb{+})$. So if $G$ is a group, then $S\pmb{-}S$ contains a syndetic subset of $G$.
\end{lem}

\begin{proof}
By Def.~\ref{def3.1}, there exist a $(\varphi,G,X)$, a point $x\in X$, an a.p. point $y\in X$ and a neighborhood $U$ of $y$ such that $\overline{G_\varphi[x]}\cap \overline{G_\varphi[y]}\not=\emptyset$ and $S=N_\varphi(x,U)$. Since $y$ is a.p. for $(\varphi,G,X)$, so $N_\varphi(y,U)$ is syndetic in $G$ and $\overline{G_\varphi[y]}\subseteq\overline{G_\varphi[x]}$. Thus, for any finite subset $A$ of $N_\varphi(y,U)$, one can take some $g\in G$ so that $g(x)$ approaches $y$ arbitrarily. Hence $\varphi(a,g(x))\in U$ for each $a\in A$. Thus $A\pmb{+}g\subseteq S$. This completes the proof.
\end{proof}

Theorem~\ref{thm1.2} proved in $\S\ref{sec2}$ will play a role in the proof of the following lemma.

\begin{lem}\label{lem3.3}
Let $S$ be a weak central set of a semi-module $(M,\pmb{+})$ over a semi-ring $(R,+,\cdot)$. Then for any finite set $F\subseteq M$ one can find a syndetic subset $N_F$ of $(R,+)$, which contains an IP-set, such that for each $d\in N_F$ there exists an element $a\in S$ with $a\pmb{+}dF\subseteq S$.
\end{lem}

\begin{proof}
Let $(\varphi,M,X)$ be a semiflow on a compact Hausdorff space $X$, $y\in X$ an a.p. point of $(\varphi,M,X)$ weakly proximal to some point $x\in X$, $U$ an open neighborhood of $y$ in $X$ such that $S=\{g\in M\,|\,\varphi(g,x)\in U\}$.

Given any finite subset $F$ of $M$, write $F=\{T_1,\dotsc,T_l\}$. Since $y$ is an a.p. point of $(\varphi,M,X)$, the orbit closure $\mathrm{cls}_X^{}M_\varphi[y]$ of $y$ is $\varphi$-minimal in $X$. For simplicity, set $Y=\mathrm{cls}_XM_\varphi[y]$.
Then by applying Theorem~\ref{thm1.2} with the minimal
$\varphi\colon M\times Y\rightarrow Y$,
it follows that:
\begin{itemize}
\item There exists some syndetic subset $N_F$ of $(R,+)$, which contains an IP-set,\footnote{This can be proved by a standard construction of Furstenberg (cf.~\cite[p.~35]{Fur}).} such that for any $d\in N_F$, one can find a point $y^\prime=y^\prime(d)\in U\cap Y$ with $\varphi(dT_i,y^\prime)\in U$ simultaneously for $i=1,2,\dots,l$.\footnote{If $(M,\pmb{+})=(R,+)$ and $\Gamma\subseteq R$ is a set of $(t_1,\dotsc,t_l)$-recurrence as in Definition~\ref{def2.8}, then $\Gamma\cap(N_F\setminus\{0\})\not=\emptyset$ with $(t_1,\dotsc,t_l)$ in place of $(T_1,\dotsc,T_l)$. Whence we can require $d\in\Gamma$.}
\end{itemize}
In addition, since $x$ is weakly proximal to $y$ under $\varphi$ and $Y$ minimal for $(\varphi,M,X)$, $Y\subseteq\mathrm{cls}_X^{}M_\varphi[x]$ and then one can find some $a\in S$ such that $\varphi(a,x)\in U$ is so close to $y^\prime$ that $\varphi(dT_i\pmb{+}a,x)\in U$ for all $1\le i\le l$.

This thus completes the proof of Lemma~\ref{lem3.3}.
\end{proof}

According to Notes~\ref{thm1.2}.1, if we consider $a\pmb{+}Fd$ instead of $a\pmb{+}dF$, then the statement of Lemma~\ref{lem3.3} still holds for any right semi-modules.

Comparing with Furstenberg's topological discussion of the case $M=\mathbb{Z}$ or $\mathbb{N}$ (cf.~\cite[Proposition~8.9]{Fur}), in the proof of Lemma~\ref{lem3.3} we have overcome the following two obstructions by using weak central set of a semi-ring and Theorem~\ref{thm1.2} proved before:
\begin{itemize}
\item Since our underlying space $X$ is not necessarily a (compact) metric space, we cannot find a Lebesgue number $\varepsilon$ and there is no the classical proximality here.
\item For our situation here, there exists no an applicable pointwise Multiple Birkhoff Recurrence Theorem (cf.~\cite{FW}, also \cite[Theorem~2.6]{Fur}).
\end{itemize}
As a central set in $\mathbb{N}$ must be a weak central set, our Lemma~\ref{lem3.3} is a generalization of Furstenberg~\cite[Proposition~8.9]{Fur}.

The following lemma is a standard result by Fursenberg's correspondence principle, which is also valid for any discrete semigroup with a zero element $\e$.

\begin{lem}\label{lem3.4}
Let $(M,\pmb{+})$ be a semi-module over a semi-ring $(R,+,\cdot)$. Then in any finite partition $M=B_1\cup\dotsm\cup B_q$ with $B_j\not=\emptyset$ for each $j=1,\dotsc,q$, one of the sets $B_j$ is a weak central set of $(M,\pmb{+})$.
\end{lem}

\begin{proof}
Over the compact Hausdorff space $X:=\{1,\dotsc,q\}^{M}$ where $\{1,\dotsc,q\}$ is discrete and each $x\in X$ is thought of as a function $x(\centerdot)\colon M\rightarrow\{1,\dotsc,q\}$, we form the semiflow
\begin{gather*}
\varphi\colon M\times X\rightarrow X\quad \textrm{or}\quad (\varphi,M,X)
\end{gather*}
in the following ways:\footnote{If $M$ is uncountable and equipped with a locally compact second countable Hausdorff topology, then $\varphi(g,x)$ is not jointly continuous with respect to $g\in M$ and $x\in X$, even not Borel measurable. Here the discrete topology of $M$ enables us to employ Furstenberg's correspondence principle.}
\begin{gather*}
(g,x(\centerdot))\mapsto\varphi(g,x(\centerdot))=x(\centerdot\pmb{+}g),\quad \forall x(\centerdot)\colon M\rightarrow\{1,\dotsc,q\}\textrm{ and }g\in M;
\end{gather*}
here $x(\centerdot\pmb{+}g)\colon M\rightarrow\{1,\dotsc,q\}$ is given by $t\mapsto x(t\pmb{+}g)$ for any $t\in M$.
Let $\xi(\centerdot)\colon M\rightarrow\{1,\dotsc,q\}$ be defined by
\begin{gather*}
\xi(g)=i\Leftrightarrow g\in B_i,\quad  i=1,\dotsc,q\textrm{ and }g\in M.
\end{gather*}
Let $\eta(\centerdot)\colon M\rightarrow\{1,\dotsc,q\}$ be an a.p. point of $(\varphi,M,X)$ by Zorn's lemma. Since the $q$ clopen blocks
\begin{gather*}
[i]_{\e}=\{x(\centerdot)\colon M\rightarrow\{1,2,\dotsc,q\}\,|\,x(\e)=i\},\quad i=1,2,\dotsc,q
\end{gather*}
form an open cover of $X$ where $\e$ is the zero element of $(M,\pmb{+})$, hence some block $[j]_{\e}$ is an open neighborhood of $\eta(\centerdot)$. Write
$S=\left\{g\in M\,|\,\varphi(g,\xi(\centerdot))\in[j]_{\e}\right\}$ which is nonempty for $B_j\not=\emptyset$.
Then $S$ is a weak central set of $(M,\pmb{+})$ by Def.~\ref{def3.1} and moreover $S=B_j$. This proves Lemma~\ref{lem3.4}.
\end{proof}

It should be noticed that although $\{1,\dotsc,q\}^{M}$ is a compact Hausdorff space, yet it is a metric space when $M$ is uncountable.
Because of this reason, we cannot employ the classical pointwise topological multiple recurrence theorem (\cite[Theorem~1.4]{FW} and \cite[Theorem~2.6]{Fur}) or the measure-theoretic multiple recurrence theorem of Furstenberg (\cite{Fur}) that are only for dynamical systems over compact metric spaces or standard Borel spaces.

\subsection{Van der Waerden-type theorems}\label{sec3.2}
Now we are able to readily prove the van der Waerden theorem of semi-modules over discrete semi-rings.

\begin{proof}[\textbf{Proof of Theorem~\ref{thm1.1}}]
The statement of Theorem~\ref{thm1.1} follows at once from Lemma~\ref{lem3.3} together with Lemma~\ref{lem3.4}.
\end{proof}

Note that applying Lemmas~\ref{lem3.4} and \ref{lem3.3} with $(R,+)$ instead of $(M,\pmb{+})$ we can see `$(2)\Rightarrow(3)$' in Theorem~\ref{thm2.9}.

Inspired by Furstenberg's concept---VDW-set in $\mathbb{Z}^m$~\cite[$\S2.4$]{Fur}, we now introduce this notion for semi-modules.

\begin{defn}\label{def3.5}
Given any $(R,+,\cdot)$-semimodule $(M,\pmb{+})$, we say that a subset $W\subseteq M$ is a \textit{van der Waerden-set} (\textit{vdW-set}) if for every finite set $F\subseteq M$ we can find a syndetic subset $D_F$ of $(R,+)$ such that for each $d\in D_F$ there is some $a\in M$ with $a\pmb{+}dF\subseteq W$.
\end{defn}

Here $D_F$ is thought of as the set of ``common differences'' of the configuration $F$. The following is another consequence of Theorem~\ref{thm1.2}.

\begin{thm}\label{thm3.6}
If $S$ is a syndetic subset of a semi-module $(M,\pmb{+})$ over a semi-ring $(R,+,\cdot)$, then $S$ is a \textit{vdW}-set.
\end{thm}

\begin{note}
Since for a general semigroup $(G,\pmb{+})$ and a finite subset $K$ of $G$
$$
(K\pmb{+}a)\cap S\not=\emptyset\ \forall a\in G\not\Rightarrow G={\bigcup}_{k\in K}(S\pmb{-}k),\quad \textrm{for here }`\pmb{-}\textrm{'}\textrm{ makes no sense}!
$$
the proof idea of \cite[Proposition~2.8]{Fur} for $R=\mathbb{Z}^m$ by van der Waerden theorem is invalid here. Differently we will prove this result by using our multiple recurrence theorem and Furstenberg's correspondence principle.
\end{note}

\begin{proof}
Let $X=\{0,1\}^M=\{x_\centerdot\colon M\rightarrow\{0,1\}\}$ with the product topology and consider the shift semiflow $\varphi\colon M\times X\rightarrow X$ given by
\begin{gather*}
g(x_\centerdot)=x_{\centerdot\pmb{+}g}\quad \forall g\in M\textrm{ and }x_\centerdot\in X.
\end{gather*}
Define a point $\xi_\centerdot$ in $X$ as follows
\begin{gather*}
\xi_g=1\Leftrightarrow g\in S,\quad \forall g\in M.
\end{gather*}
Since the orbit closure $\mathrm{cls}_X^{}M[\xi_\centerdot]$ is $M$-invariant, we can find a minimal point $y_\centerdot\in\mathrm{cls}_X^{}M[\xi_\centerdot]$ for $(\varphi,M,X)$; i.e., $\mathrm{cls}_X^{}M[y_\centerdot]$ is minimal for $(\varphi,M,X)$ such that $y_\centerdot\in\mathrm{cls}_X^{}M[y_\centerdot]$. As $S$ is ``syndetic'' in $M$ associated to some finite subset, say $K=\{g_1,\dotsc,g_k\}\subseteq M$, it follows that there exists some element $\hat{g}\in M$ such that
\begin{gather*}
\xi_{g_1\pmb{+}\hat{g}}=y_{g_1}, \dotsc,\xi_{g_k\pmb{+}\hat{g}}=y_{g_k}\quad \textrm{and then  }1\in\{y_{g_1},\dotsc,y_{g_k}\}
\end{gather*}
Without loss of generality, assume $y_{g^\prime}=1$. Since the cylinder set
\begin{gather*}
U=\{x_\centerdot\in X\,|\,x_{g^\prime}=1\}
\end{gather*}
is an open neighborhood of $y_\centerdot$ in $X$, then by Theorem~\ref{thm1.2}, it follows that:
\begin{itemize}
\item For any finite set $F\subseteq M$ there exists a syndetic set $D_F\subseteq (R,+)$ such that for any $d\in D_F$ one can choose some point $z_\centerdot\in U\cap\mathrm{cls}_X^{}M[\xi_\centerdot]$ with $z_\centerdot=1$ on $dF$. \footnote{If $(M,\pmb{+})=(R,+)$ and $\Gamma$ is a set of $(t_1,\dotsc,t_l)$-recurrence, then we can additionally require that $\Gamma\cap(D_F\setminus\{0\})\not=\emptyset$ for any $F=\{t_1,\dotsc,t_l\}\subseteq R$. This is what needed for proving `$(1)\Rightarrow(4)$' in Theorem~\ref{thm2.9}.}
\end{itemize}
This implies that there is some $a\in M$ so that $\xi_{|a\pmb{+}dF}\equiv 1$. Hence $a\pmb{+}dF\subseteq S$.

This thus proves Theorem~\ref{thm3.6}.
\end{proof}

We now present an application of our van der Waerden-type result Theorem~\ref{thm1.1}.

\begin{thm}\label{thm3.7}
Let $\Lambda$ be a compact metric space, $(M,\pmb{+})$ a semi-module over a semi-ring $(R,+,\cdot)$, and let $f\colon M\rightarrow\Lambda$ be an arbitrary function. Then for any $\varepsilon>0$ and finite set $F\subseteq M$, we can find a syndetic subset $D$ of $(R,+)$ such that for any $d\in D$ there exists an $a\in M$
so that $f(a\pmb{+}dF)$ is a set of diameter less than $\varepsilon$ in $\Lambda$.
\end{thm}

\begin{proof}
This follows from Theorem~\ref{thm1.1} by an argument same as that of \cite[Theorem~2.9]{Fur} for $G=\mathbb{N}^m$. So we omit the details here.
\end{proof}

Theorem~\ref{thm3.7} implies that if $f\colon G\rightarrow\mathbb{R}$ is a bounded function on a semigroup $G$ and $\varepsilon>0$, then there will be three elements in ``arithmetic progression'' $a, a+h, a+2h$ in $G$ such that $|f(a)-f(a+h)|<\varepsilon$ and $|f(a+h)-f(a+2h)|<\varepsilon$.

\begin{cor}
Let $X$ be an arbitrary space and $T_1,\dotsc,T_l$ commuting transformations of $X$ to itself, and let $\phi_1,\dotsc,\phi_l$ be any functions from $X$ to the unit circle $\mathbb{T}\colon |\phi_i(x)|=1$. Then for any $\varepsilon>0$ and any $m\in\mathbb{N}$, we can find a syndetic subset $D$ of $\mathbb{Z}_+$ such that to any $n\in D$ there is an $x_n\in X$ to satisfy the inequalities:
\begin{gather*}
\left|{\phi_i(T_1^{k_1n}T_2^{k_2n}\dotsm T_l^{k_ln}x_n)}^{k_0n}-1\right|<\varepsilon\quad \forall i=1,\dotsc,l
\end{gather*}
for all $(k_0,k_1,\dotsc,k_l)\in{\mathbb{Z}}_+^{l+1}$ with $0\le k_j\le m$ for each $j=0,1,\dotsc,l$.
\end{cor}

\begin{proof}
Given any point $x_0\in X$, we now define the vector-valued function $\Phi\colon\mathbb{Z}_+^{l+1}\rightarrow\mathbb{T}^l$ by setting
\begin{gather*}
\Phi(n_0,n_1,\dotsc,n_l)=\left(\begin{matrix}{\phi_1(T_1^{n_1}T_2^{n_2}\dotsm T_l^{n_l}x_0)}^{n_0}\\{\phi_2(T_1^{n_1}T_2^{n_2}\dotsm T_l^{n_l}x_0)}^{n_0}\\\vdots\\{\phi_l(T_1^{n_1}T_2^{n_2}\dotsm T_l^{n_l}x_0)}^{n_0}\end{matrix}\right)\quad \forall (n_0,n_1,\dotsc,n_l)\in\mathbb{Z}_+^{l+1}.
\end{gather*}
Choosing a metric on $\mathbb{T}^l$ by $\|(z_1,\dotsc,z_l)-(y_1,\dotsc,y_l)\|=\max_{1\le i\le l}|z_i-y_i|$, we proceed by Theorem~\ref{thm3.7} with $(M,\pmb{+})=(\mathbb{Z}_+^{l+1},\pmb{+})$ over the semi-ring $(\mathbb{Z}_+,+,\cdot)$ to find a ``syndetic'' subset $D$ of $\mathbb{Z}_+$ associated to $\varepsilon>0$ and the $(l+1)$-dimensional cube
$$
F=\left\{(k_0,k_1,\dotsc,k_l)\in\mathbb{Z}_+^{l+1}\,\big{|}\,0\le k_j\le m, 0\le j\le l\right\}
$$
such that for any $n\in D$, one can find some element $a=(n_0,n_1,\dotsc,n_l)\in\mathbb{Z}_+^{l+1}$ with
\begin{gather*}
\textrm{diam}\left(\Phi(a\pmb{+}nF)\right)<\varepsilon.
\end{gather*}
If we now set
$$
x_n=T_1^{n_1}T_2^{n_2}\dotsm T_l^{n_l}x_0
$$
and compare the values of $\Phi$ at the vertices of the homothetic copy $a\pmb{+}nF$ of the cube $F$, we can see that for any $i=1,2,\dotsc,l$ and all $(k_0,k_1,\dotsc,l_l)\in F$,
\begin{gather*}
\left|{\phi_i(T_1^{k_1n}\dotsm T_l^{k_ln}x_n)}^{n_0}-{\phi_i(T_1^{k_1n}\dotsm T_l^{k_ln}x_n)}^{n_0+k_0n}\right|<\varepsilon\quad \textrm{or}\quad\left|{\phi_i(T_1^{k_1n}\dotsm T_l^{k_ln}x_n)}^{k_0n}-1\right|<\varepsilon.
\end{gather*}
The proof is completed.
\end{proof}

It should be noted here that this result is just a strengthen of a theorem of Furstenberg~\cite[Theorem~2.13]{Fur}.

For convenience we now introduce a minor technical condition:
\begin{itemize}
\item[$(\ast)$] We shall say $(R,+,\cdot)$ is an \textit{$*$-semiring} if for any $s_1,\dotsc,s_k\in R$ where $1\le k<\infty$,
\begin{gather*}
\mathfrak{N}_{s_1}\cup\dotsm\cup\mathfrak{N}_{s_k}\not=R,\quad \textrm{where }\mathfrak{N}_s:=\{t\in R\,|\,s+t=0\}.
\end{gather*}
\end{itemize}
It is easy to see that if $(R,+,\cdot)$ is an infinite ring (therefore $(R,+)$ is an additive group), then it is an $\ast$-semiring; moreover, if $(R,+)$ is cancelable infinite, then condition $(*)$ holds.

Clearly $(\mathbb{Z}_+^n,+,\cdot)$ and $(\mathbb{R}_+^n,+,\cdot)$ both are commutative $*$-semirings with $0=(0,\dotsc,0)$ and $1=(1,\dotsc,1)$ where
\begin{gather*}
(x_1,\dotsc,x_n)+(y_1,\dotsc,y_n)=(x_1+y_1,\dotsc,x_n+y_n)\intertext{and}(x_1,\dotsc,x_n)\cdot(y_1,\dotsc,y_n)=(x_1y_1,\dotsc,x_ny_n)
\end{gather*}
for all $(x_1,\dotsc,x_n)$ and $(y_1,\dotsc,y_n)\in\mathbb{R}_+^n$. On the other hand, let $\mathbb{R}_+^{n\times n}$ be the set of all real $n\times n$ nonnegative matrices; then $(\mathbb{R}_+^{n\times n},+,\circ)$ is an $*$-semiring cancelable noncommutative.

Clearly if $S$ is a syndetic subset of an $\ast$-semiring $(R,+,\cdot)$, then one can always find some element $d\in S$ with $d\not=0$ because for all $t\in R\setminus(\mathfrak{N}_{s_1}\cup\dotsm\cup\mathfrak{N}_{s_k})$ where $K=\{s_1,\dotsc,s_k\}$, $(K+t)\cap S\not=\emptyset$.

Now we let $(M,\pmb{+})$ be a semi-module over a $*$-semiring $(R,+,\cdot)$. The following variation of Theorem~\ref{thm1.1} is just the ``finitary formulation'' in which one considers partitions of large sets (but finite if $G$ is countable).

\begin{thm}\label{thm3.9}
Let $M_1\subset M_2\subset\dotsm\subset M_n\subset\dotsm$ with $M=\bigcup_{n\ge1}M_n$; and let $F$ be a finite subset of $M$ and $q\in\mathbb{N}$. Then there exists a number $N=N(q,F)$ such that whenever $n\ge N$ and $M_n=B_1\cup\dotsm\cup B_q$ is a partition into $q$ sets, one of these $B_j$ contains a homothetic copy of $F$, $a\pmb{+}dF$, where $a\in M$ and $d\in R$ with $d\not=0$.
\end{thm}

\begin{proof}
We think of the partition $M_n=B_1\cup\dotsm\cup B_q$ as a function $\xi_n$ from $M_n$ to $\{1,2,\dotsc,q\}$ defined by $\xi_n(g)=j\Leftrightarrow g\in B_j$ for $j=1,\dotsc,q$, for any $n\ge1$. Suppose with $n\to+\infty$ we can find partitions for which no homothetic copy of $F$ is contained in any cell $B_j$. Consider the corresponding function from $M_n$ to $\{1,2,\dotsc,q\}$ and extend it arbitrarily onto $M$ to obtain a point $\xi_n\in\{1,2,\dotsc,q\}^M$. Take any limit point of $\{\xi_n\}$, say $\xi$, and apply Theorem~\ref{thm1.1} to the corresponding partition
\begin{gather*}
M=\{\xi=1\}\cup\dotsm\cup\{\xi=q\}.
\end{gather*}
It follows that $\xi$ is constant on some homothetic copy $a\pmb{+}dF$ where $a\in M$ and $d\in R$ with $d\not=0$. This set $a\pmb{+}dF$ is contained in $M_n$ as soon as $n$ is large, and moreover $\xi_n$ agrees with $\xi$ on $a\pmb{+}dF$ for some large $n$. But this clearly leads to a contradiction that proves the theorem.
\end{proof}

This theorem generalizes obviously the classical result \cite[Theorem~2.10]{Fur} for $\mathbb{N}^m$. It is clear that since $M$ need not be countable here, this version does not imply our previous formulation Theorem~\ref{thm1.1} in general.

\subsection{An open problem}
Finally we conclude our discussion with the following open question closely related to our topic discussed before:

\begin{con}[Schur-Brauer version of van der Waerden's theorem]\label{con3.10}
Let $M=B_1\cup\dotsm\cup B_q$ be any finite partition of a semi-module $(M,\pmb{+})$ over a semi-ring $(R,+,\cdot)$. One of the sets $B_j$ has the property that if $F$ is any finite subset of $R$, then there are elements $a\in M$ and $b\in B_j$ with $b\not=\e$ such that $a\pmb{+}Fb\subseteq B_j$.
\end{con}

Notice here that `$F\subseteq R$' in Conjecture~\ref{con3.10}, but not `$F\subseteq G$' as in Theorem~\ref{thm1.1} there. If $G=\mathbb{Z}_+$, this is just the Schur-Brauer theorem.
\section*{\textbf{Acknowledgments}}%
This work was partly supported by National Natural Science Foundation of China (Grant Nos. 11431012, 11271183) and PAPD of Jiangsu Higher Education Institutions.


\end{document}